\documentclass[11pt, a4paper]{amsart} 
\usepackage[left = 2.3cm, right = 2.30cm, top = 2.5cm, bottom = 2.5cm]{geometry}

\usepackage[utf8]{inputenc}
\usepackage[english]{babel}
\usepackage[T1]{fontenc}
\usepackage{lmodern}

\usepackage{amsmath, amsfonts, amssymb,  mathabx, mathrsfs, mathtools, dsfont}
\usepackage{rotating}
\usepackage{subcaption}
\usepackage{graphicx}

\usepackage{hyperref}
\usepackage[capitalize,noabbrev]{cleveref}
\hypersetup{%
	pdftitle={Convergence and error analysis of a semi-implicit finite volume scheme for the Gray--Scott system},  									
	colorlinks=true, 								
	linkcolor= red,									
	anchorcolor=black,
	citecolor= red,									
	urlcolor = red									
	pdfauthor = {Tsiry Avisoa Randrianasolo},
	pdfcreator = {Tsiry Avisoa Randrianasolo},
	pdfkeywords = {IMEX schemes, finite volume methods, reaction-diffusion equations, Gray--Scott model, pattern formation}
}
\newcommand{\la}{\langle}
\newcommand{\ra}{\rangle}
\newcommand{\dd}{\,\mathrm{d}}
\newcommand{\RR}{\mathbb{R}}
\newcommand{\mC}{\mathcal{C}}
\newcommand{\mT}{\mathcal{T}}

\newcommand{\NN}{\mathbb{N}}
\newcommand{\HH}{\mathbb{H}}
\newcommand{\VV}{\mathbb{V}}
\newcommand{\mI}{\mathcal{I}}
\DeclareMathOperator*{\esssup}{{esssup}}
\DeclareMathOperator{\sech}{sech}
\newtheorem{theorem}{Theorem}[section]
\newtheorem{lemma}[theorem]{Lemma}%
\newtheorem{proposition}[theorem]{Proposition}%
\newtheorem{definition}{Definition}[section]
\newtheorem{algo}{Algorithm}[section]
\numberwithin{equation}{section}
\Crefname{proposition}{Proposition}{Propositions}
\title[Finite volume method for the Gray--Scott system]{Convergence and error analysis of a semi-implicit finite volume scheme for the Gray--Scott system}
\author{Tsiry Avisoa Randrianasolo}
\address{Faculty of Mathematics, Bielefeld University, 33615 Bielefeld, Germany}
\email{trandria@math.uni-bielefeld.de}
\thanks{This work was funded by the Deutsche Forschungsgemeinschaft (DFG, German Research Foundation) -- Project-ID 317210226 - SFB 1283.}

\begin{document}
	%
	%
	\begin{abstract}
		We analyze a semi-implicit finite volume scheme for the Gray--Scott system, a model for pattern formation in chemical and biological media. We prove unconditional well-posedness of the fully discrete problem and establish qualitative properties, including positivity and boundedness of the numerical solution. A convergence result is obtained by compactness arguments, showing that the discrete approximations converge strongly to a weak solution of the continuous system. Under additional regularity assumptions, we further derive a priori error estimates in the $L^2$ norm. Numerical experiments validate the theoretical analysis, confirm a rate of convergence of order 1, and illustrate the ability of the scheme to capture classical Gray--Scott patterns.
	\end{abstract}
	\maketitle

	%
	%
	\section{Introduction}\label{sec:intro}

	We consider the numerical approximation of the \emph{Gray--Scott reaction-diffusion system} on a bounded domain $D \subset \RR\times\RR$ with Lipschitz boundary. This system models the evolution of two interacting chemical species with concentrations $u \coloneqq u(t,x)$ and $v \coloneqq v(t,x)$, governed by the nonlinear system:
	\begin{equation} \label{eq:gs}
		\left\{
		\begin{alignedat}{6}
			\partial_t u &= d_u \Delta u &&- uv^2 &&+ F(1 - u), 
			\\
			\partial_t v &= d_v \Delta v &&+ uv^2 &&- (F + k)v,
		\end{alignedat}
		\right.
	\end{equation}
	for $(t,x) \in (0,T)\times D$, supplemented with initial conditions $u(0,x) = u_0(x)$, $v(0,x) = v_0(x)$. The constants $d_u>0$ and $d_v>0$ are the diffusion coefficients for the chemical species, respectively; 
	$F>0$ represents the feed rate of the reactant $u$, and $k>0$ denotes the removal rate of the intermediate species $v$.
	In addition, we impose homogeneous Neumann boundary conditions on $\partial D$, modeling a closed or insulated system in which no material can enter or leave the domain. In biological terms, this models scenarios such as a petri dish or a membrane-bounded cell where neither species can escape the domain.
	
	The Gray--Scott model, originally introduced by P.\ Gray and S.K.\ Scott in the 1980s, see e.g.\ \cite{gray1983autocatalytic}, is derived as a simplification of the Oregonator model of the Belousov--Zhabotinsky reaction. The Gray--Scott popularity as a model system grew significantly following the influential work of Pearson \cite{pearson1993complex}, who demonstrated through numerical simulations that even minimal reaction-diffusion systems like Gray--Scott can give rise to surprisingly complex and diverse patterns. As discussed in texts such as Epstein and Pojman's Introduction to Nonlinear Chemical Dynamics \cite{epstein1998introduction}, the model is also known to exhibit both deterministic chaos and spatial patterning.
	As such, beyond its chemical origins, the Gray--Scott system serves as a paradigmatic example in mathematical biology, physics, and materials science, where it illustrates emergent behavior, self-organization, and instability-driven pattern formation.
	
	Concerning the mathematical well-posedness. The Gray--Scott model belongs to a broader class of reaction-diffusion systems whose mathematical properties, such as existence, uniqueness, and regularity of solutions, have been extensively studied using both weak and strong solution frameworks.
	
	For initial data $u_0,v_0\in L^2(D)$, one can construct global-in-time weak solutions using classical techniques, including energy estimates, monotonicity methods, and compactness arguments (see, e.g., \cite{amann1989dynamic, pierre2010global}). These solutions are defined in the distributional sense and satisfy associated energy inequalities. When the nonlinearities satisfy structural conditions such as quasi-positivity, cooperativity, and subcritical growth, global existence is ensured under homogeneous Neumann or Dirichlet boundary conditions.
	
	If the initial data are further assumed to lie in $L^\infty(D)$, then the solutions exhibit improved regularity: they remain uniformly bounded in $L^\infty(D)$ for all times, and standard bootstrapping and maximum principle arguments show that they are also continuous in space and time. In this case, one obtains global strong solutions, which satisfy the system almost everywhere and possess additional spatial regularity.
	
	An intermediate case is when $u_0,v_0\in H^1(D)$. In this setting, the solution enjoys enhanced regularity, namely
	\[
	u, v \in \mC([0,T];L^2(D)) \cap L^2(0,T;H^1(D)).
	\]
	This follows from standard parabolic regularity theory, as shown, e.g., in Amann's analysis of quasilinear parabolic systems \cite[Theorem 2.1]{amann1989dynamic}, which demonstrates continuity in time and Sobolev regularity under minimal assumptions on the initial data. Although framed for general quasilinear systems, these results apply directly to the semilinear structure of the Gray--Scott model. A similar conclusion is shown in Pierre's survey \cite{pierre2010global}, where the regularizing effect of the diffusion operator is shown to yield such space-time regularity for initial data in $H^1(D)$.
	
	Even in the minimal case, where $u_0,v_0\in L^2(D)$, the smoothing property of the heat semigroup ensures that for any $t>0$, the solution becomes essentially bounded in space
	\[
	u, v \in L^\infty(D), \quad\mbox{ for all } t\geq t_0 >0,
	\]
	as established, for example, in \cite{souplet2006global}. Therefore, under general assumptions, bounded domain $D\subset\RR\times\RR$ with Lipschitz boundary, quasi-positive nonlinearities, and initial data in $L^\infty(D)$, one can show that
	\[
	u, v \in L^\infty(0,\infty;L^\infty(D)),
	\]
	with further regularity determined by the spatial dimension, the form of the nonlinearity, and the choice of boundary conditions. These regularity properties are not only central to the theoretical understanding of the system but also essential in the numerical modeling of the system.
	
	On the numerical side, several classes of numerical schemes have been proposed for the Gray--Scott system, each motivated by the nonlinear and stiff nature of the equations:
	
	Zhang, Wong, and Zhang~\cite{zhang2008second} considered finite element discretizations, constructing a second-order implicit-explicit (IMEX) Galerkin finite element scheme, treating the diffusion terms implicitly and the nonlinear reaction terms explicitly. Their analysis established optimal $L^2$ error estimates and confirmed second-order accuracy in both space and time. Hansen~\cite{hansen2015a} proposed more abstract convergence results for operator-splitting methods applied to semilinear evolution equations that includes the Gray--Scott model as a canonical example. Their stabilized finite element formulations mitigates spurious oscillations, particularly in parameter regimes that generate sharp patterns. More recently, Haggar, Mahamat Malloum, Fokam, and Mbehou~\cite{haggar2025a} introduced a linearized two-step finite element scheme combining Crank--Nicolson with a second-order Backward Differentiation Formula also called BDF2. Their error-splitting analysis proves unconditional stability and optimal convergence in both $L^2$ and $H^1$ norms.
	
	Another important line of research is based on the Method Of Lines approach. In this approach, the spatial variables are discretized (e.g., by finite differences or finite elements), transforming the Gray--Scott PDE system into a large system of stiff Ordinary Differential Equations (ODEs) in time. Once this semi-discrete system is obtained, one can directly apply standard ODE solvers. In \cite{boscarino2016high}, Boscarino, Filbet, and Russo used high-order IMEX Runge--Kutta methods, where the stiff linear diffusion terms are treated implicitly while the nonlinear reaction terms are handled explicitly. This strategy combines the stability advantages of implicit ODE solvers with the efficiency of explicit treatment of nonlinearities, leading to schemes with favorable stability properties and high-order accuracy suitable for multiscale reaction-diffusion dynamics.
	
	Together, these works illustrate the diversity of available discretization strategies, but most analyses focus on finite element or spectral frameworks coupled with IMEX time integration. By contrast, finite volume schemes, despite their popularity in applications due to local conservation and their natural fit to Cartesian grids, have received far less rigorous mathematical attention for the Gray--Scott system.
	
	In this work, we develop and rigorously analyze a fully discrete finite volume scheme for the Gray--Scott system. The method is based on a first-order IMEX discretization, where diffusion is treated implicitly and nonlinear reaction terms explicitly. We prove unconditional well-posedness of the discrete problem and establish key qualitative properties, including non-negativity and boundedness of the numerical solution. Using compactness arguments combined with weak-strong uniqueness, we show that the fully discrete approximation converges strongly in $L^2$-norm to a weak solution of the continuous system. Under additional regularity assumptions, we further derive error estimates in the $L^2$-norm. Finally, numerical experiments confirm the theoretical results and reproduce classical Gray--Scott patterns..
	
	The remainder of this paper is organized as follows. In \cref{sec:prelim}, we introduce the functional setting and notation; and define the notion of weak solution for the Gray--Scott system; and summarizes its essential analytical properties, including existence, regularity, and uniqueness. In \cref{sec:discretization}, we describe the spatial and temporal discretization using a finite volume method and formulate the fully discrete semi-implicit scheme. \Cref{sec:main_results} is devoted to the main results. There, we establish that the discrete solution converges to the weak solution of the continuous problem, and derive error estimates under suitable regularity assumptions on the exact solution. Finally, \cref{sec:numerics} presents numerical experiments that validate the theoretical results and illustrate the pattern formation behavior of the model. We conclude in \cref{sec:conclusion} with a discussion of potential extensions, including data assimilation and feedback control.

	%
	%
	\section{Preliminaries}\label{sec:prelim}
	
	In this section, we introduce the mathematical framework and notations employed throughout the paper. We also recall key well-posedness results for the Gray--Scott system \eqref{eq:gs}.
	
	\subsection{Functional settings and notations}\label{subsec:settings}

	We denote by $L^p\coloneqq L^p(D)$, for $1 \leq p \leq \infty$, the usual Lebesgue spaces of measurable functions defined on $D$. The associated norm is written as
	\[
	\Vert u \Vert_{L^p}^p \coloneqq   \int_D \vert u(x)\vert^p \dd x  , \quad \text{for } 1 \leq p < \infty, \quad\Vert u \Vert_{L^\infty} \coloneqq \esssup_{x \in D} \vert u(x)\vert .
	\]
	
	For $s\in\NN$ and $1 \leq p \leq \infty$, the Sobolev space $W^{s, p}(D)$ consists of functions whose weak derivatives up to order $s$ belong to  $L^p(D)$. When $p = 2$, we write $H^s\coloneqq H^s(D)= W^{s,p}(D)$, which is a Hilbert space with inner product and norm
	\[
	((u,v))_\alpha\coloneqq \sum_{\vert\alpha\vert\leq s} (\partial^\alpha u,\partial^\alpha v),\quad\Vert u\Vert_{H^s}^2\coloneqq \sum_{\vert\alpha\vert\leq s} \Vert \partial^\alpha u\Vert ^2_{L^2} ,
	\]
	where $(\,\cdot\,,\,\cdot\,)$ denotes the standard inner product in $L^2$. 
	
	We consider solutions subject to homogeneous Neumann boundary conditions. Since such problems determine solutions only up to an additive constant, we work with the mean-zero subspaces
	$$
	\HH\coloneqq\bigg\{u\in L^2:\int_D u(x)\dd x  = 0\bigg\},\quad \VV\coloneqq\bigg\{u\in H^1:\int_D u(x)\dd x  = 0\bigg\},
	$$
	and let $\VV'\coloneqq H^{-1}$ denote  the dual space of $\VV$;   by $\la\,\cdot\,,\,\cdot\,\ra$, the duality pairing between $\VV$ and $\VV'$.
	
	For time-dependent functions, we write $L^p(0,T;X)$ to denote Bochner spaces of $X$-valued functions that are $L^p$-integrable in time, where $X$ is a Banach space.

	%
	%
	\subsection{Weak Formulation}\label{subsec:weak_formulation}
	We now define the notion of weak solution for the Gray--Scott system \eqref{eq:gs} and state the assumptions under which existence and uniqueness are guaranteed.

	\begin{definition}[Weak solution]\label{def:weak_solution}
		Let $T>0$. A pair of functions $(u,v)$ such that
		\[
		u, v \in L^2(0,T; \VV) \cap L^\infty(0,T;\HH), \quad 
		\partial_t u, \partial_t v \in L^2(0,T; \VV')
		\]
		is a weak solution to the Gray--Scott system \eqref{eq:gs} if for all test functions $\phi, \psi \in \VV$ and almost every $t \in (0, T)$, the equations
		\begin{alignat}{5}
			\label{eq:weak_v}
			\la \partial_t u(t), \phi \ra + d_u (\nabla u(t), \nabla \phi ) 
			&=    (&-u(t)v^2(t) &+ F(1 - u(t)), \phi)& ,  
			\\
			\label{eq:weak_u}
			\la \partial_t v(t), \psi \rangle + d_v  (\nabla v(t), \nabla \psi) 
			&= (&u(t)v^2(t) &- (F + k)v(t), \psi) &
		\end{alignat}
		hold with initial data satisfying
		\begin{equation}\label{eq:ic}
			u_0,v_0\in L^\infty,\quad\mbox{ with }\quad 0\leq u_0(x)\leq 1,\quad v_0(x)\geq 0\quad\mbox{a.e. in }D. 
		\end{equation}
	\end{definition}

	The condition \eqref{eq:ic} reflects the physical interpretation of $u$ and $v$ as concentrations, and ensures that the solution remains in a physically meaningful range.
	
	Under the above assumptions, global-in-time weak solutions exist and remain nonnegative and bounded. Specifically, one has
	\[
	0 \leq u(t,x) \leq 1 \quad \text{and} \quad v(t,x) \geq 0 \quad \text{for almost every } (t,x) \in [0,T] \times D,
	\]
	and moreover,
	\[
	u, v \in L^\infty(0,T;L^\infty).
	\]
	These results follow from the general theory of reaction-diffusion systems with quasi-positive nonlinearities and bounded initial data (see, e.g., \cite{amann1989dynamic,pierre2010global, souplet2006global}). In particular, the Gray--Scott system satisfies the structural assumptions required for the application of maximum principles, comparison techniques, and a priori bounds. The quasi-positivity of the reaction terms ensures that nonnegative initial data yields nonnegative solutions, while boundedness is propagated by the nonlinear structure and diffusion.
	
	Moreover, if the initial data $u_0, v_0 \in H^1$, then the solution enjoys enhanced regularity, notably
	\[
	u, v \in \mC([0,T];L^2) \cap L^2(0,T;H^1).
	\]
	This follows from classical parabolic regularity theory, relying on energy estimates and the smoothing effect of the diffusion operator.
	Uniqueness of weak solutions follows from standard monotonicity and Gronwall-type arguments (see, e.g., \cite{henry1981geometric, pierre2010global}).

	%
	%
	%
	%
	\section{Semi-implicit finite volume discretization}\label{sec:discretization}
	
	In this section, we introduce the fully discrete finite volume scheme used to approximate solutions of the Gray--Scott system \eqref{eq:gs}. The method is based on a spatial discretization over a uniform mesh and a first-order semi-implicit time integration scheme.
	
	\subsection{Finite volume discretization of the spatial domain}\label{subsec:mesh}

	Let $\mT_h = \{K\}$ denote a finite collection of non-overlapping square control volumes of side length $h$, forming a uniform Cartesian mesh of $D$. Each cell $K \in \mathcal{T}_h$  is associated with a representative point $x_K \in K$, chosen as the geometric center of $K$, so that the cell average $u_K \approx u(x_K)$ for all $K \in \mathcal{T}_h$.
	
	To discretize the diffusion operator, we follow the finite volume framework outlined in \cite{eymard2000finite}. 
	Integrating the Laplacian over each control volume and applying the divergence theorem gives
	\[
	\int_K \Delta u \dd x = \int_{\partial K} \nabla u \cdot n_K \dd S \approx \sum_{\sigma \subset \partial K} \Phi_\sigma,
	\]
	where $n_K$ is the outward unit normal to $\partial K$, and $\Phi_\sigma$ is the numerical flux across the face $\sigma$. 
	
	We use the two-point flux approximation to define the flux as
	\[
	\Phi_\sigma = - \vert \sigma\vert \frac{u_L - u_K}{d_{KL}},
	\]
	where $u_K$ and $u_L$ are the values at the centers of adjacent cells $K$ and $L$, and $d_{KL} = \vert x_L - x_K\vert$. This leads to the discrete approximation
	\[
	\int_{\partial K} \nabla u \cdot n_K \dd S \approx - \sum_{\sigma \subset \partial K} \tau_\sigma (u_L - u_K),
	\]
	with transmissibility coefficient $\tau_\sigma =  {\vert\sigma\vert}/{d_{KL}}$. For a uniform mesh, $d_{KL} = h$.
	
	We define the discrete space $\HH_h \subset L^2(D)$ of piecewise constant functions,
	\[
	\HH_h \coloneqq \bigg\{ u_h \in\HH: u_h|_K = \bigg(\frac{1}{h^2}\int_K u(y)\dd y\bigg) \text{ for all } K \in \mathcal{T}_h\bigg \}.
	\]
	For $w_h, \phi_h \in \HH_h$, the discrete inner products are defined as 
	\[
	\big(w_h, \phi_h\big)_h \coloneqq h^2 \sum_{K \in \mathcal{T}_h} w_K \phi_K,\quad
	\big(\nabla w_h, \nabla \phi_h\big)_h \coloneqq \sum_{\substack{K, L \in \mathcal{T}_h \\ \sigma = K|L}} \tau_{\sigma} (w_K - w_L)(\phi_K - \phi_L).
	\]

	Given this setup, we define the cellwise average  interpolant $\mI_h: \VV\rightarrow \HH_h$, such that
	\begin{equation*}
		\mI_h u\coloneqq\sum_{K\in\mT_h}\bigg(\frac{1}{h^2}\int_K u(y)\dd y\bigg)\chi_{K},\quad\forall u\in \VV,
	\end{equation*}
	where $\chi_{K}$ denotes the characteristic function of $K$.
	
	To estimate the error introduced by the interpolant $\mI_h$, we recall the following result from finite element theory, which also holds in the finite volume framework, see \cite[Theorem 4.4.4 ]{brenner2008mathematical}:
	\begin{proposition}[Bramble--Hilbert lemma]\label{prop:bramble_hilbert_lemma}
		Let $K\subset D$ be a bounded Lipschitz domain of diameter $\mathrm{diam}(K)$ and $m\geq 0$ be an integer. Let $\phi:H^1(K)\rightarrow L^2(K)$ be a bounded linear functional that vanishes on all polynomials in $\mathbb{P}_m$ (i.e., of degree $\leq m$). Then, for all $w\in H^s(K)$,
		\begin{equation*}
			\Vert \phi(w)\Vert_{L^2(K)}\leq C\mathrm{diam}(K)^{s-m}\Vert w\Vert_{H^s(K)},
		\end{equation*}
		with a constant $C$ that depends only on the shape of $K$ and the choice of the functional $\phi$.
	\end{proposition}
	
	Using this, we state the following interpolation estimate:
	\begin{lemma}\label{lem:approx_pty}
		For any $u\in H^1$, the following inequality holds
		\begin{equation*}
			\Vert u-\mI_h u\Vert_{L^2}^2\leq \gamma_0h^2\Vert\nabla u\Vert^2_{L^2},
		\end{equation*}
		where the constant $\gamma_0$ depends only on the reference cell geometry and the quadrature rule.
	\end{lemma}
	\begin{proof}
		Let $u\in \VV$. Because $\mI_Hu$ is a piecewise constant
		\begin{align*}
			\Vert u-\mI_h u\Vert_{L^2}^2&= \sum_{K\in\mT_h} \Big\Vert u - \frac{1}{h^2}\int_K u(y)\dd y \Big\Vert_{L^2(K)}^2.
		\end{align*}
		Since the interpolant integrates constants exactly, the associated error functional
		\begin{equation*}
			\phi(u)\coloneqq u - \frac{1}{h^2}\int_K u(y)\dd y
		\end{equation*}
		vanishes on the constants of each $K$. Thus, Proposition~\ref{prop:bramble_hilbert_lemma} applies directly, yielding
		\begin{equation*}
			\Big\Vert u - \frac{1}{h^2}\int_K u(y)\dd y \Big\Vert_{L^2(K)}^2\leq Ch^{2}\Vert u\Vert_{H^1(K)}^2.
		\end{equation*}
		Summing over $K$,  we obtain the global estimate, as claimed.
	\end{proof}
	
	%
	%
	\subsection{The fully discrete scheme}\label{subsec:fully_discrete_scheme}
	
	We now define the fully discrete numerical scheme. 
	
	The time interval $[0,T]$ is divided into $N$ uniform time steps of size $\Delta t > 0$, with $t^n = n\Delta t$ for $n = 0, \dots, N$, where $t^0 = 0$, and $t^N = T$. At each time step $t^n$, let $u_h^n, v_h^n \in \VV_h$ denote the approximations of $u(t^n, \cdot)$ and $v(t^n, \cdot)$.

	%
	%
	%
	\begin{algo}\label{algo:scheme}
		For $n = 0,\ldots,N-1$: 
		
		Find $u_h^{n+1}, v_h^{n+1} \in \VV_h$ such that
		\begin{alignat}{6}
			\label{eq:discrete_gs_vfu_0}
			\big(u_h^{n+1} - u_h^n,\, &\phi_h \big)_h + \Delta t\, d_u \big(\nabla u_h^{n+1},\, &\nabla\phi_h\big)_h
			&= \Delta t \,\big(&-u_h^n (v_h^n)^2 &+ F(1 - u_h^n),\, &\phi_h &\big)_h,  
			\\
			\label{eq:discrete_gs_vfv_0}
			\big(v_h^{n+1} - v_h^n,\, &\psi_h\big)_h + \Delta t\, d_v \big(\nabla v_h^{n+1},\, &\nabla\psi_h\big)_h
			&= \Delta t\, \big(&u_h^n (v_h^n)^2 &- (F+k)v_h^n,\, &\psi_h &\big)_h, 
		\end{alignat}
		for all   $\phi_h, \psi_h \in \VV_h$.
	\end{algo}
	%
	%
	This scheme uses an implicit treatment of the diffusion operator together with an explicit discretization of the nonlinear reaction terms. Such an IMEX formulation ensures stability in the presence of stiff dynamics, while preserving the simplicity and local conservation properties that are characteristic of the finite volume method.

	%
	%
	\section{Main results}\label{sec:main_results}
	Let $ (u_h^n, v_h^n )\in \VV_h\times \VV_h $ be the solution to the Gray--Scott system~\eqref{eq:gs} as computed by \cref{algo:scheme}.
	For all $t\in (t^{n},t^{n+1}]$, we define piecewise linear in time interpolants
	\begin{align*}
		u_h^{\Delta t}(t) = \frac{t^{n+1} - t}{\Delta t} u_h^n + \frac{t - t^{n}}{\Delta t} u_h^{n+1},
		\quad
		v_h^{\Delta t}(t)\coloneqq \frac{t^{n+1} - t}{\Delta t} v_h^n + \frac{t - t^{n}}{\Delta t} v_h^{n+1}.
	\end{align*}
	
	\begin{theorem}[Convergence]\label{thm:convergence}
		As $h,\Delta t\to 0$, the sequences $u_h^{\Delta t}, v_h^{\Delta t}$ defined above converge to the functions $u,v$, which satisfy the weak formulation of the Gray--Scott system as defined by Definition~\ref{def:weak_solution}. The convergence
		is understood in the following sense:
		\begin{alignat*}{6}
			&u_h^{\Delta t}, &\, &v_h^{\Delta t}&\to{}& u, &\, &v &&\quad\mbox{strongly in } &&L^2(0,T;\HH),
			\\
			&\nabla u_h^{\Delta t} ,&\, &\nabla v_h^{\Delta t} &{}\rightharpoonup{}& \nabla u ,&\, &\nabla v &&\quad\mbox{weakly in } &&L^2(0,T;\HH\times \HH).
		\end{alignat*}
	\end{theorem}
	The proof of \cref{thm:convergence} is postponed in \cref{subsec:convergence}.
	
	\begin{theorem}[Error estimate]\label{thm:error_estimate}
		With further regularity, i.e.,
		\[
		u, v \in L^\infty(0,T; H^2\cap\HH),\quad \partial_t u, \partial_t v \in L^2(0,T; \VV),\quad \partial_{tt} u, \partial_{tt} v \in L^\infty(0,T; \VV');
		\]
		we have
		\begin{equation*}
			\max_{0\leq n\leq N} \Big(\Vert u_h^n - u(t^n) \Vert_{\HH} + \Vert v_h^n - v(t^n) \Vert_{\HH}\Big)\leq C(h^2 + \Delta t),
		\end{equation*}
		where the constant $C>0$ depends on the domain $D$, terminal time $T$, and norms of the exact solution $u, v$, but is independent of $h$, $\Delta t$, or $n$.
	\end{theorem}
	The proof of \cref{thm:error_estimate} is postponed in \cref{subsec:error_estimate}.
	
	\subsection{\texorpdfstring{Proof of the convergence~\cref{thm:convergence}}{Proof of the convergence result}}\label{subsec:convergence}

	Let  $a_h^u, a_h^v : \VV_h \times \VV_h \rightarrow \mathbb{R}$ be bilinear forms defined as
	\[
	a_h^u(u,\phi_h) \coloneqq \big(u, \,\phi_h\big)_h + \Delta t\, d_u \big( \nabla u,\, \nabla\phi_h\big)_h, \quad
	a_h^v(v,\psi_h) \coloneqq \big(v, \,\psi_h\big)_h + \Delta t\, d_v \big(\nabla v,\, \nabla\psi_h\big)_h,
	\]
	and the associated linear functionals:
	\begin{alignat*}{4}
	 	\ell_h^u(\phi_h) &\coloneqq \big(u_h^n,\, \phi_h \big)_h + \Delta t\, \big(&-u_h^n (v_h^n)^2 &+ F(1 - u_h^n),\, &\phi_h \big)_h, 
		\\
		\ell_h^v(\psi_h) &\coloneqq \big(v_h^n,\, \psi_h \big)_h + \Delta t\, \big(& u_h^n (v_h^n)^2 &- (F + k) v_h^n, \,&\psi_h \big)_h.
	\end{alignat*}
	
	Then the variational formulation of the fully discrete problem reads:  
	
	Find $u_h^{n+1}, v_h^{n+1} \in \VV_h$ such that
	\begin{align}
		a_h^u(u_h^{n+1}, \phi_h) &= \ell_h^u(\phi_h), \quad \forall \phi_h \in \VV_h, \label{eq:discrete_gs_vfu} 
		\\
		a_h^v(v_h^{n+1}, \psi_h) &= \ell_h^v(\psi_h), \quad \forall \psi_h \in \VV_h. \label{eq:discrete_gs_vfv}
	\end{align}
	
	\begin{lemma}\label{lem:discrete_gs_unique_solution}
		For $n = 0,\ldots,N-1$, the variational problem \eqref{eq:discrete_gs_vfu}  (resp.\ \eqref{eq:discrete_gs_vfv}) admits a unique solution $u_h^{n+1}\in \VV_h$ (resp.\ $v_h^{n+1}\in \VV_h$).
	\end{lemma}
	\begin{proof}
		We verify the conditions of the Lax--Milgram theorem:
		
		\textit{Coercivity.} For all $u\in \VV_h$, we have
		\begin{equation*}
			a_h^u(u,u)=  \Vert u\Vert_{L^2}^2 + \Delta t\, d_u\big(\nabla u,\, \nabla u\big)\geq  \Vert u\Vert_{\VV}^2.
		\end{equation*}
		Thus, $a_h^u(u,u) \geq \alpha\Vert u\Vert_{\VV}^2$ with $\alpha = 1 $. The same holds for $a_h^v$.
		
		\medskip 
		
		\textit{Continuity.} We have
		\begin{align*}
			\vert a_h^u(u, \phi)\vert & \leq C \Vert u \Vert_{\HH} \Vert \phi \Vert_{\HH},
		\end{align*}
		for some constant $C>0$. Again the same applies to $a_h^v$.
		
		\medskip
		
		\textit{Boundedness of linear forms} $\ell_h^u$ and $\ell_h^v$. We estimate
		\begin{equation*}
			\ell_h^u(\phi_h)\leq \Big( \Vert u_h^n\Vert_{L^2} + \Delta t\, \Vert u_h^n(v_h^n)^2 \Vert_{L^2}  + \Delta t\, F(1 +\Vert u_h^n\Vert_{L^2})\Big) \Vert \phi_h\Vert_{L^2}.
		\end{equation*}
		The same holds for $\ell_h^v$.
		Thus, $\ell_h^u$ and $\ell_h^v$ are bounded linear functionals on $\VV_h$.
		
		\medskip
		
		The conditions of the Lax--Milgram theorem are satisfied. Thus, there exist a unique solution $u_h^{n+1}$ (resp.\ $v_h^{n+1}$) in $\VV_h$ to \eqref{eq:discrete_gs_vfu} (resp.\ \eqref{eq:discrete_gs_vfv}), as claimed by the lemma.
	\end{proof}

	\begin{lemma}\label{lem:discrete_gs_bounded_solution}
		Let $n\in\{0,\ldots,N-1\}$. Assume that $0\leq u_h^n\leq 1$ and $0\leq v_h^n\leq v_{\mathrm{max}}$ a.e.\ in $D$. Then, it holds for a.e.\ in $D$ that
		\[
		0\leq u_h^{n+1}\leq 1,\quad 0\leq v_h^{n+1}\leq v_{\mathrm{max}}.
		\]
	\end{lemma}
	\begin{proof}
		We fix $n\in\{0,\ldots,N-1\}$ and assume that $0\leq u_h^n\leq 1$ and $0\leq v_h^n\leq v_{\mathrm{max}}$ a.e.\ in $D$. We proceed with a Stampacchia truncation
		applied to both $u_h^n$ and $v_h^n$.
		
		\textit{Non-negativity.} 
		Let $u_h^{n+1,-}\coloneqq\min\{0,u_h^{n+1}\}\in \VV_h$ and $v_h^{n+1,-}\coloneqq\min\{0,v_h^{n+1}\}\in \VV_h$. We take $\phi_h = u_h^{n+1,-}$ in \eqref{eq:discrete_gs_vfu} and $\psi_h = v_h^{n+1,-}$ in \eqref{eq:discrete_gs_vfv}. By symmetry and coercivity of $a_h^u$ and $a_h^v$, we have
		\begin{alignat}{3}
			\label{eq:gs_nonneg_u}
			0&\leq a_h^u(u_h^{n+1,-},u_h^{n+1,-}) &{}=a_h^u(u_h^{n+1},u_h^{n+1,-}) &=\ell_h^u(u_h^{n},&u_h^{n+1,-}),
			\\
			\label{eq:gs_nonneg_v}
			0 &\leq a_h^v(v_h^{n+1,-},v_h^{n+1,-}) &{}= a_h^v(v_h^{n+1},v_h^{n+1,-}) &= \ell_h^v(v_h^{n},&v_h^{n+1,-}).
		\end{alignat}
		The right-hand side of \eqref{eq:gs_nonneg_u} (resp.\ \eqref{eq:gs_nonneg_v}) are non-positive when $u_h^{n} = 0$ (resp.\ $v_h^{n} = 0$). That is a consequence of the quasi-positivity of the reaction terms. Meanwhile, the left-hand sides are non-negative. Necessarily, $\Vert u_h^{n+1,-}\Vert_{L^2} =\Vert v_h^{n+1,-}\Vert_{L^2} = 0$, so, $u_h^{n+1}, v_h^{n+1}\geq 0$ a.e.\ in $D$. 
		
		\textit{Upper bound.} For all $x\in\RR$, we denote $x^+\coloneqq\max\{0,x\}$. Let $y_h\coloneqq (u_h^{n+1}- 1)^+\in \VV_h$ and $z_h\coloneqq (v_h^{n+1}- v_{\mathrm{max}})^+\in \VV_h$. We take $\phi_h = y_h$ in \eqref{eq:discrete_gs_vfu} and $\psi_h = z_h$ in \eqref{eq:discrete_gs_vfv}. Again, by symmetry and coercivity, it follows that
		\begin{alignat}{3}
			\label{eq:gs_nonneg_u_1}
			\Vert y_h\Vert_{L^2}(\Vert y_h\Vert_{L^2} - 1)\leq  a_h^u(u_h^{n+1},y_h) &= \ell_h^u(u_h^{n},&y_h),
			\\
			\label{eq:gs_nonneg_v_1}
			\Vert z_h\Vert_{L^2}(\Vert z_h\Vert_{L^2} - v_{\mathrm{max}})\leq  a_h^v(v_h^{n+1},z_h) &= \ell_h^v(v_h^{n},&z_h).
		\end{alignat}
		The right-hand side of \eqref{eq:gs_nonneg_u_1} (resp.\ \eqref{eq:gs_nonneg_v_1}) are non-positive when $u_h^{n} = 1$ (resp.\ $v_h^{n} = 0$).
		That implies $\Vert y_h\Vert_{L^2} = 0$ (resp.\ $\Vert z_h\Vert_{L^2} = 0$), so $u_h^{n+1}\leq 1$ (resp.\ $v_h^{n+1}\leq v_{\mathrm{max}}$) a.e.\ in $D$. That concludes the proof of the lemma.
	\end{proof}

	\begin{lemma}\label{lem:discrete_gs_apriori}
		For all $N\in\NN$, it holds that
		\begin{align*}
			\max_{0\leq n\leq N}\Big(\Vert u_h^n\Vert_{L^2}^2 + \Vert v_h^n\Vert_{L^2}^2 \Big)+ \Delta t \sum_{n = 1}^{N}\Big(d_u\Vert u_h^n\Vert_{H^1}^2 + d_v\Vert v_h^n\Vert_{H^1}^2\Big)&\leq C,
			\\
			\sum_{n = 1}^{N}\Big(\Vert u_h^{n} - u_h^{n-1}\Vert_{L^2}^2+\Vert v_h^{n} - v_h^{n-1}\Vert_{L^2}^2\Big)&\leq C,
		\end{align*}
		with $C\coloneqq (D, F, v_{\mathrm{max}}, T)>0$.
	\end{lemma}
	\begin{proof}
		We take $\phi_h = 2u_h^{n+1}$ in \eqref{eq:discrete_gs_vfu_0} and $\psi_h = 2v_h^{n+1}$ in \eqref{eq:discrete_gs_vfv_0}; and use the identity
		$
		(a - b) 2b = a^2 - b^2 + (a-b)^2
		$. Then, using \cref{lem:discrete_gs_bounded_solution}, and the Cauchy-Schwarz and Young inequalities
		\begin{alignat*}{5}
			&\big(-&u_h^n (v_h^n)^2& + F(1 - u_h^n),\, 2u_h^{n+1} \big)&{}\leq{} C + \Vert u_h^{n+1}\Vert_{L^2}^2,
			\\
			&\big(&u_h^n (v_h^n)^2 & - (F + k)v_h^n,\, 2v_h^{n+1} \big)&{}\leq{} C + \Vert v_h^{n+1}\Vert_{L^2}^2,
		\end{alignat*}
		with $C\coloneqq C(D,v_{\mathrm{max}})>0$. After these calculations, we arrive at
		\begin{alignat*}{4}
			&&\Vert u_h^{n+1}\Vert_{L^2}^2 &- \Vert u_h^{n}\Vert_{L^2}^2& + \Vert u_h^{n+1} - u_h^{n}\Vert_{L^2}^2& + \Delta t \,d_u\Vert u_h^{n+1}\Vert_{H^1}^2 
			&{}\leq  \Delta t C+ \Delta t\Vert u_h^{n+1}\Vert_{L^2}^2,
			\\[5pt]
			&&\Vert v_h^{n+1}\Vert_{L^2}^2 &- \Vert v_h^{n}\Vert_{L^2}^2& + \Vert v_h^{n+1} - v_h^{n}\Vert_{L^2}^2& + \Delta t \, d_v\Vert v_h^{n+1}\Vert_{H^1}^2 
			&{}\leq  \Delta t C + \Delta t\Vert v_h^{n+1}\Vert_{L^2}^2.
		\end{alignat*}
		Summing for $n = 0,\ldots,N-1$,
		\begin{alignat*}{8}
			&\Vert u_h^{N}\Vert_{L^2}^2 &+ \sum_{n = 1}^{N}\Vert u_h^{n} - u_h^{n-1}\Vert_{L^2}^2 &+ \Delta t \,d_u\sum_{n = 1}^{N}\Vert u_h^{n}\Vert_{H^1}^2& {}\leq{} &  \Vert u_h^{0}\Vert_{L^2}^2 &+ TC &+ \Delta t\sum_{n = 1}^{N}\Vert u_h^{n}\Vert_{L^2}^2,
			\\
			&\Vert v_h^{N}\Vert_{L^2}^2 &+ \sum_{n = 1}^{N}\Vert v_h^{n} - v_h^{n-1}\Vert_{L^2}^2 & + \Delta t \,d_v\sum_{n = 1}^{N}\Vert v_h^{n}\Vert_{H^1}^2 
			&{}\leq {}  &\Vert v_h^{0}\Vert_{L^2}^2 &+ T C &+ \Delta t\sum_{n = 1}^{N}\Vert v_h^{n}\Vert_{L^2}^2.
		\end{alignat*}
		By the Gronwall inequality,
		\begin{alignat*}{3}
			&\Vert u_h^{N}\Vert_{L^2}^2 &+ \sum_{n = 1}^{N}\Vert u_h^{n} - u_h^{n-1}\Vert_{L^2}^2 &+ \Delta t\, d_u\sum_{n = 1}^{N}\Vert u_h^{n}\Vert_{H^1}^2 &{}\leq   C,
			\\
			&\Vert v_h^{N}\Vert_{L^2}^2 &+ \sum_{n = 1}^{N}\Vert v_h^{n} - v_h^{n-1}\Vert_{L^2}^2 & + \Delta t\, d_v\sum_{n = 1}^{N}\Vert v_h^{n}\Vert_{H^1}^2 &{}\leq   C,
		\end{alignat*}
		with $C\coloneqq C(D, F, v_{\mathrm{max}}, T)>0$.
		
		This concludes the proof of the lemma.
	\end{proof}

	\begin{proof}[Proof of \cref{thm:convergence}]
		
		Since $	u_h^{\Delta t},	v_h^{\Delta t}$ are defined by linear interpolation in time between $u_h^n,v_h^n\in \VV_h$, and by the discrete energy estimate in Lemma~\ref{lem:discrete_gs_apriori}, we have that $\partial_t u_h^{\Delta t},\partial_t v_h^{\Delta t}\in L^2(0,T; \HH_h)$.
		
		By Aubin--Lions lemma; there exists a subsequence; which we still denote by $u_h^{\Delta t},v_h^{\Delta t}$; and limits $u,v\in L^2(0,T;\HH)$; such that
		\begin{alignat}{7}
			\label{eq:strong_L2}
			&u_h^{\Delta t}, &\, &v_h^{\Delta t}&\to{}& u, &\, &v &&\quad\mbox{strongly in } &L^2(0,T;\HH),
			\\
			\label{eq:weak_H1}
			&u_h^{\Delta t}, &\, &v_h^{\Delta t} &\rightharpoonup{}& u, &\,  &v &&\quad\mbox{weakly in } &L^2(0,T;\VV),
			\\
			\label{eq:weak_time}
			&\partial_t u_h^{\Delta t}, &\, &\partial_t v_h^{\Delta t} &\rightharpoonup{}& \partial_tu, &\,  &\partial_t v &&\quad\mbox{weakly in } &L^2(0,T; \HH).
		\end{alignat}
		
		By interpolation of the discrete problems \eqref{eq:discrete_gs_vfu_0} and \eqref{eq:discrete_gs_vfv_0} over time, the interpolants $u_h^{\Delta t}$ and $v_h^{\Delta t}$ satisfy also
		\begin{alignat*}{4}
			&\int_0^T\big(\partial_t u_h^{\Delta t}(s), \phi_h\big)_h &+  d_u \big(\nabla u_h^{\Delta t}(s),\,\nabla \phi_h\big)_h \dd s
			&{}= \int_0^T\big( &-u_h^{\Delta t}(s) (v_h^{\Delta t})^2(s) &+ F(1 - u_h^{\Delta t}(s)), \,&\phi_h \big)_h\dd s ,
			\\
			&\int_0^T\big(\partial_t v_h^{\Delta t}(s), \psi_h\big)_h &+ d_v\big(\nabla v_h^{\Delta t}(s),\nabla \psi_h\big)_h\dd s
			&{}= \int_0^T\big( &u_h^{\Delta t}(s) (v_h^{\Delta t})^2(s) &- (F+k)v_h^{\Delta t}(s),\, &\psi_h \big)_h\dd s .
		\end{alignat*}
		
		To pass to the limit in the nonlinear terms, we use the identity
		\[
		u_h^{\Delta t} (v_h^{\Delta t})^2 - u v^2 = (u_h^{\Delta t} - u)(v_h^{\Delta t})^2 + u ((v_h^{\Delta t})^2 - v^2), 
		\]
		and apply H\"older inequality to get
		\begin{align*}
			\big\Vert u_h^{\Delta t} (v_h^{\Delta t})^2 - u v^2\big\Vert _{L^1(0,T;L^1)} \leq{} &\big\Vert u_h^{\Delta t}\big\Vert _{L^2(0,T;L^2)} \big\Vert (v_h^{\Delta t})^2 - v^2\big\Vert _{L^2(0,T;L^2)} 
			\\
			&+ \big\Vert v^2\Vert _{L^2(0,T;L^2)} \big\Vert u_h^{\Delta t} - u\big\Vert _{L^2(0,T;L^2)}.
		\end{align*}
		By \eqref{eq:strong_L2}, both terms on the right-hand side tend to 0 when $h,\Delta t\to 0$, thus 
		\[
		u_h^{\Delta t}(v_h^{\Delta t})^2 \to uv^2 \mbox{ strongly in } L^1(0,T;L^1(D)).
		\]
		Finally, passing to the limit in the weak formulation, we obtain
		\begin{alignat*}{4}
				\int_0^T \big(\partial_t u(s),\, \phi\big) + d_u \big(  \nabla u(s),  \, \nabla \phi\big) \dd s&= \int_0^T \big(&-u(s)v^2(s) + F(1 - u(s)), \,&\phi\big)\dd s, 
				\\
				\int_0^T \big(\partial_t v(s), \,\psi\big) + d_v \big(  \nabla v(s),\,   \nabla \psi\big)\dd s &= \int_0^T \big(&u(s)v^2(s) - (F + k)v(s), \,&\psi\big)\dd s,
		\end{alignat*}
		for all test functions $\phi, \psi\in\VV$.
		
		Hence, $(u,v)$ is the unique weak solution of the Gray--Scott system.
		
		This completes the proof of \cref{thm:convergence}.
	\end{proof}

	\subsection{\texorpdfstring{Proof of the error estimate~\cref{thm:error_estimate}}{Proof of the error estimate}}\label{subsec:error_estimate}
	
	In this section, we derive an a priori error estimate for the proposed \cref{algo:scheme}, under the assumption that the exact solution is sufficiently regular.
	
	To establish error estimates, we compare the exact solution with its discrete approximation at the level of their weak formulations. Specifically, we test the continuous weak form \eqref{eq:weak_v}-\eqref{eq:weak_u} against discrete functions in $\VV_h$ and subtract the finite volume scheme \eqref{eq:discrete_gs_vfu_0}-\eqref{eq:discrete_gs_vfv_0}, which give
	\begin{align*}
		\begin{split}
			\Big\la   \frac{\mI_h u(t^{n+1}) -  \mI_h u(t^{n})}{\Delta t}, \, \phi_h \Big\ra &+  d_u \big(\nabla \mI_h u(t^{n+1}), \nabla \phi_h \big)
			\\ 
			&=   \big(-\mI_h (uv^2)(t^{n}) + F(1 - \mI_h u(t^{n})),\, \phi_h\big) + R_h^n(\phi_h),
		\end{split}  
		\\[5pt]
		\begin{split}
			\Big\la   \frac{\mI_h v(t^{n+1}) - \mI_h v(t^{n})}{\Delta t},\, \psi_h \Big\ra &+  d_v  \big(\nabla \mI_h v(t^{n+1}), \nabla \psi_h\big) 
			\\
			&=   \big(\hphantom{-}\mI_h(uv^2)(t^{n}) - (F + k)\mI_h v(t^{n}),\, \psi_h\big) + Q_h^n(\psi_h).
		\end{split}
	\end{align*} 
	This procedure isolates the consistency error, which appears as residual terms 
	\begin{alignat*}{5}
		R_h^n(\phi_h) &\coloneqq& \Big( \frac{u(t^{n+1}) - u(t^n)}{\Delta t},\, \phi_h \Big)_h 
		&+ d_u \big( \nabla u(t^{n+1}),\, &\nabla \phi_h \big)_h 
		&-  \big( &-u(t^n)v^2(t^n) &+ F(1 - u(t^n)),\, &\phi_h \big )_h, 
		\\
		Q_h^n(\psi_h) &\coloneqq &\Big( \frac{v(t^{n+1}) - v(t^n)}{\Delta t},\, \psi_h \Big)_h 
		&+ d_v \big( \nabla v(t^{n+1}),\, &\nabla \psi_h \big)_h 
		&-  \big(&u(t^n)v^2(t^n) &- (F + k)v(t^n),\, &\psi_h  \big)_h.
	\end{alignat*} 
		
	As a very first step in the proof, we show that both terms remain small.	
	\begin{lemma}\label{lem:residuals}
		Assume that the hypotheses of \cref{thm:error_estimate} hold. Then, the residuals satisfy  
		\begin{equation*}
			\max_{0\leq n\leq N} \Big(\Vert R_h^n(\phi_h)\Vert_{H^{-1}} + \Vert Q_h^n(\psi_h)\Vert_{H^{-1}}\Big)\leq C(h^2 + \Delta t),\quad\forall \phi_h,\psi_h\in \VV_h,
		\end{equation*}
		where $C>0$ is a constant that depends on the norms of $u,v$, but independent on the scheme.
	\end{lemma}
	\begin{proof}
		We give the proof for $R_h^n$; the case of $Q_h^n$ follows analogously.
		Let $\phi_h\in \VV$. To estimate the action of $R_h^n$ on $\phi_h$, we insert the weak form and subtract it, then bound the dual norm. We split the residual into 3 terms,
		\begin{equation*}
			R_h^n(\phi_h) = T_1 + T_2 + T_3,
		\end{equation*}
		with  
		\begin{align*}
			\intertext{the \textit{time truncation}}
			T_1&\coloneqq \Big( \frac{u(t^{n+1}) - u(t^n)}{\Delta t} - \partial_t u(t^{n+1}),\, \phi_h \Big)_h  +  \big(\partial_t u(t^{n+1}),\, \phi_h\big)_h - \big\la\partial_t u(t^{n+1}),\, \phi_h \big\ra,
			\\
			\intertext{  the \textit{gradient quadrature error}}
			T_2&\coloneqq d_u \big( \nabla u(t^{n+1}), \,\nabla \phi_h \big)_h- d_u \big( \nabla u(t^{n+1}), \,\nabla \phi_h \big),
			\vphantom{\Big( \frac{u(t^{n+1}) - u(t^n)}{\Delta t} - \partial_t u(t^{n+1}),\, \phi_h \Big)_h}
			\\
			\intertext{and the \textit{nonlinear quadrature error}}
			T_3&\coloneqq \big( f(u(t^n),v(t^n)), \, \phi_h \big)_h - \big(f(u(t^n),v(t^n)),\, \phi_h \big).
			\vphantom{\Big( \frac{u(t^{n+1}) - u(t^n)}{\Delta t} - \partial_t u(t^{n+1}), \,\phi_h \Big)_h}
		\end{align*}

		\textit{Time truncation.} 
		Using Taylor expansion,
		\begin{equation*}
			\frac{u(t^{n+1}) - u(t^n)}{\Delta t} - \partial_t u(t^{n+1})  = \frac{\Delta t}{2}\partial_{tt}u(\xi),\quad\mbox{ for some } \xi\in (t^{n}, t^{n+1}),
		\end{equation*}
		which gives for the first term
		\begin{equation*}
			\bigg\Vert \frac{u(t^{n+1}) - u(t^n)}{\Delta t} - \partial_t u(t^{n+1}) \bigg\Vert_{H^{-1}}\leq C\Vert \partial_{tt}u\Vert_{L^\infty(t^n,t^{n+1}; H^{-1})}.
		\end{equation*}
		We use \cref{lem:approx_pty} on the second term (quadrature error for $L^2$-inner product), which gives
		\begin{equation*}
			\big\vert  \big(\partial_t u(t^{n+1}),\, \phi_h\big)_h - \big\la\partial_t u(t^{n+1}),\, \phi_h \big\ra\big\vert \leq Ch^2\big\Vert \partial_t u(t^{n+1})\big\Vert_{H^1}\Vert \phi_h\Vert_{H^1}.
		\end{equation*}
		Thus, the first contribution is
		\begin{equation*}
			\vert T_1\vert\leq C(\Delta t + h^2) \Vert \phi_h\Vert_{H^1}.
		\end{equation*}
		
		\textit{Gradient quadrature error.} Since $\nabla u(t^{n+1})\in H^1\times H^1$, and  we use midpoint rule (or piecewise constant projection), by \cref{lem:approx_pty}, the quadrature error satisfies
		\begin{equation*}
			\big\vert   \big( \nabla u(t^{n+1}), \nabla \phi_h \big)_h-   \big( \nabla u(t^{n+1}), \nabla \phi_h \big)\big\vert \leq C h^2\Vert \nabla u(t^{n+1})\Vert_{H^{1}}\Vert \phi_h\Vert_{H^1}.
		\end{equation*}
		Thus, 
		\begin{equation*}
			\vert T_2\vert\leq Ch^2\Vert \phi_h\Vert_{H^{1}}.
		\end{equation*}
		
		\textit{Nonlinear quadrature error.} The reaction term $f(u,v) = -uv^2 + F(1-u)\in H^1$ if $u,v\in H^2$. Similarly, by applying \cref{lem:approx_pty}, we have
		\begin{equation*}
			\big\vert  \big(f(u(t^n),v(t^n)), \, \phi_h \big)_h - \big(f(u(t^n),v(t^n), \, \phi_h\big)\big\vert \leq Ch^2\Vert f(u,v)\Vert_{H^1}\Vert \phi_h\Vert_{H^{1}}.
		\end{equation*}
		Thus, 
		\begin{equation*}
			\vert T_3\vert \leq Ch^2\Vert \phi_h\Vert_{H^{1}}.
		\end{equation*}
		
		Combining the estimates, we arrive at
		\begin{equation*}
			\vert R_h^n(\phi_h)\vert \leq C(\Delta t + h^2)\Vert \phi_h\Vert_{H^{1}},
		\end{equation*}
		thus, by definition of the dual norm
		\begin{equation*}
			\Vert R_h^n(\phi_h)\Vert_{H^{-1}}\leq C(h^2 + \Delta t),
		\end{equation*}
		where $C>0$ is a constant depends on the norms of $u$, but not on the scheme.
		
		That completes the proof of the lemma.
	\end{proof}
	\begin{proof}[Proof of \cref{thm:error_estimate}]\label{proof:error_estimate}
			We define the error terms:
		\[
		\delta_h^n := u_h^n - \mI_h u(t^n), \qquad \eta_h^n := v_h^n - \mI_h v(t^n).
		\]
		We subtract the 2 systems of equations to get
		\begin{alignat*}{11}
				&\big\la   \delta_h^{n+1} - & \delta_h^{n},\,& \phi_h \big\ra + \Delta t\, d_u \big(\nabla \delta_h^{n+1},\, &\nabla \phi_h \big) &={}&    \Delta t\, \big(f(u_h^h,v_h^n)- &f(\mI_h u(t^n),&\,\mI_h v(t^n)),\, &\phi_h\big) &+& R_h^n(\phi_h),
				\\[5pt]
				&\big\la \eta_h^{n+1} -& \eta_h^{n},\,& \psi_h \big\ra +\Delta t\,  d_v  \big(\nabla \eta_h^{n+1},\, &\nabla \psi_h\big) 
				&={}& \Delta t\, \big(g(u_h^h,v_h^n)- &g(\mI_h u(t^n),&\,\mI_hv(t^n)), \,&\psi_h\big) &+& Q_h^n(\psi_h),
		\end{alignat*}
		where we define $f(u,v)\coloneqq -uv^2 + F(1-u)$ and $g(u,v)\coloneqq uv^2 - (F+k)v$.

		We test with $\phi_h = 2\delta_h^{n+1}$ and $\psi_h = 2\eta_h^{n+1}$; and use the identity $
		(a - b) 2b = a^2 - b^2 + (a-b)^2$,
		\begin{align*}
			\begin{split}
				\Vert \delta_h^{n+1}\Vert_{L^2}^2 -  \Vert \delta_h^{n}\Vert_{L^2}^2 &+\Vert \delta_h^{n+1} - \delta_h^{n}\Vert_{L^2}^2 + \Delta t\, d_u \Vert \nabla \delta_h^{n+1}\Vert_{L^2}^2
				\\ 
				&=    2\Delta t \,\big(f(u_h^h,v_h^n)- f(\mI_h u(t^n),\mI_h v(t^n)),\, \delta_h^{n+1}\big) + R_h^n(2\delta_h^{n+1}),
			\end{split}  
			\\[5pt]
			\begin{split}
				\Vert \eta_h^{n+1}\Vert_{L^2}^2 - \Vert \eta_h^{n}\Vert_{L^2}^2 &+	\Vert \eta_h^{n+1} - \eta_h^{n}\Vert_{L^2}^2+\Delta t \, d_v  \Vert \nabla \eta_h^{n+1}\Vert_{L^2}^2 
				\\
				&= 2\Delta t\, \big(g(u_h^h,v_h^n)- g(\mI_h u(t^n),\mI_h v(t^n)),\, \eta_h^{n+1}\big) + Q_h^n(2\eta_h^{n+1}).
			\end{split}
		\end{align*}
		We can write 
		\begin{alignat*}{3}
			&\big\vert \big(f(u_h^h,v_h^n)- &f(\mI_h u(t^n),\mI_h v(t^n)), \, \delta_h^{n+1}\big)_h\big\vert &\leq C\big( \Vert \delta_h^{n}\Vert_{L^2}^2 + \Vert \eta_h^{n}\Vert_{L^2}^2 + \tfrac12\Vert \delta_h^{n+1}\Vert_{L^2}^2\big),
			\\
			&\big\vert (g(u_h^h,v_h^n)- &g(\mI_h u(t^n),\mI_h v(t^n)), \,\eta_h^{n+1}\big)_h\big\vert &\leq C\big( \Vert \delta_h^{n}\Vert_{L^2}^2 + \Vert \eta_h^{n}\Vert_{L^2}^2 + \tfrac12\Vert \eta_h^{n+1}\Vert_{L^2}^2\big).
		\end{alignat*}
		Next, we apply Cauchy--Schwarz and Young inequalities on the residuals; and use Lemma~\ref{lem:residuals},
		\begin{alignat*}{3}
			&R_h^n(2\delta_h^{n+1})&{}\leq{} &2\Vert R_h^n\Vert_{H^{-1}}\Vert \delta_h^{n+1}\Vert_{H^{1}}&{}\leq {}&C(h + \Delta t)^2 + \varepsilon \Vert \delta_h^{n+1}\Vert_{H^{1}}^2,
			\\
			&Q_h^n(2\eta_h^{n+1})&{}\leq {}&2\Vert Q_h^n\Vert_{H^{-1}}\Vert \eta_h^{n+1}\Vert_{H^{1}}&{}\leq {}&C(h + \Delta t)^2 + \varepsilon \Vert \eta_h^{n+1}\Vert_{H^{1}}^2.
		\end{alignat*}
		
		Combining everything, we arrive at
		\begin{alignat*}{3}
				&\Vert \delta_h^{n+1}\Vert_{L^2}^2 -  \Vert \delta_h^{n}\Vert_{L^2}^2 + \Delta t\, d_u \Vert \nabla \delta_h^{n+1}\Vert_{L^2}^2 &\leq    C\Delta t\big( \Vert \delta_h^{n}\Vert_{L^2}^2 + \Vert \eta_h^{n}\Vert_{L^2}^2 \big) + C(h^2 + \Delta t)^2,
			\\
				&\Vert \eta_h^{n+1}\Vert_{L^2}^2 - \Vert \eta_h^{n}\Vert_{L^2}^2 + \Delta t \,d_v \Vert \nabla \eta_h^{n+1}\Vert_{L^2}^2 &\leq   C\Delta t\big( \Vert \delta_h^{n}\Vert_{L^2}^2 + \Vert \eta_h^{n}\Vert_{L^2}^2 \big) + C(h^2 + \Delta t)^2.
			\end{alignat*}
		We define $\mathcal{E}^n\coloneqq \Vert \delta_h^{n}\Vert_{L^2}^2 + \Vert \eta_h^{n}\Vert_{L^2}^2$, and obtain a recurrence
		\begin{align*}
			\mathcal{E}^{n+1}\leq (1 + C\Delta t)\mathcal{E}^n + C(h^2 + \Delta t)^2,
		\end{align*}
		which by the discrete Gronwall lemma, provides
		\begin{equation*}
			\max_{0\leq n\leq N} \Big(\Vert u_h^n - u(t^n) \Vert_{L^2} + \Vert v_h^n - v(t^n) \Vert_{L^2}\Big)\leq C(h^2 + \Delta t)
		\end{equation*}
		as claimed by \cref{thm:error_estimate}.
	\end{proof}

	%
	%
	\section{Numerical experiments}\label{sec:numerics}
	
		\begin{figure}[t]
		\begin{subfigure}[c]{1\textwidth}
			\centering
			\makebox[.1\textwidth][l]{}
			\hfill
			\makebox[.2\textwidth][l]{\large $100$ tu}
			\hfill
			\makebox[.2\textwidth][l]{\large $500$ tu}
			\hfill
			\makebox[.2\textwidth][l]{\large $1000$ tu}
			\hfill
			\makebox[.2\textwidth][l]{\large $2000$ tu}
			\par\smallskip
			
			\begin{minipage}[t]{0.02\textwidth}
				\rotatebox{90}{\textbf{Labyrinthine}}
			\end{minipage}
			\includegraphics[scale = 0.4]{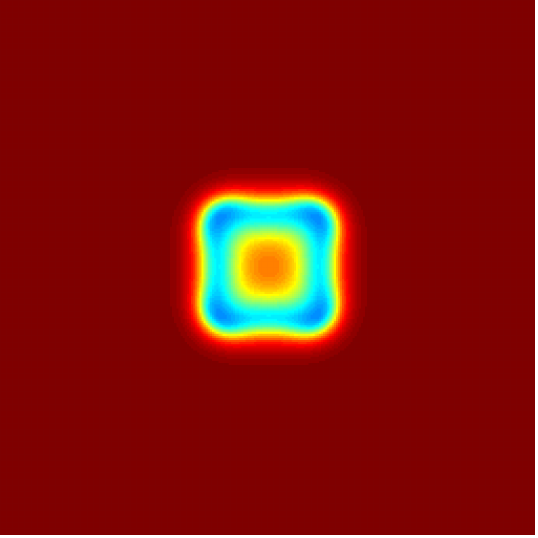}
			\includegraphics[scale = 0.4]{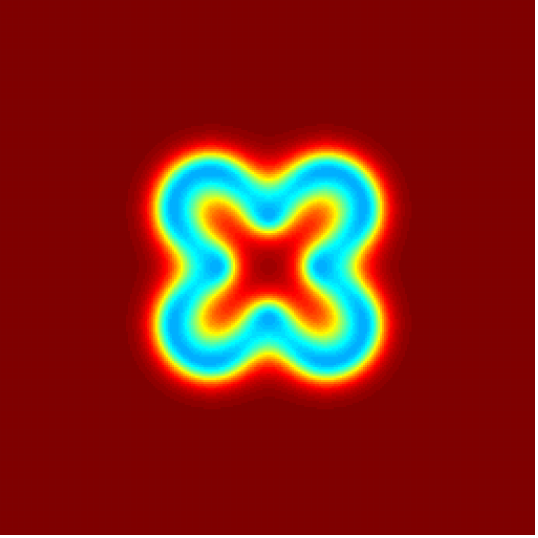}
			\includegraphics[scale = 0.4]{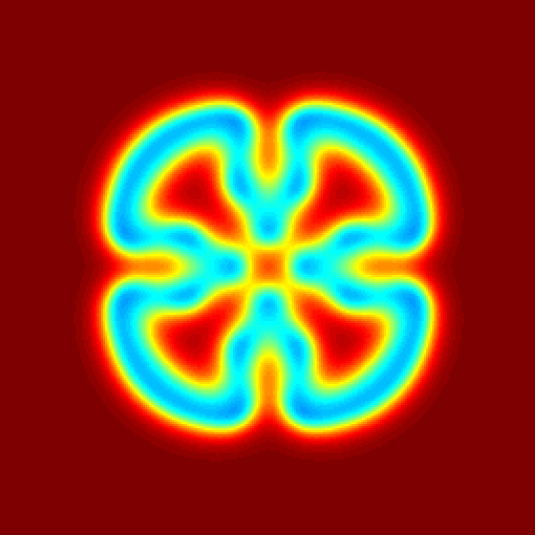}
			\includegraphics[scale = 0.4]{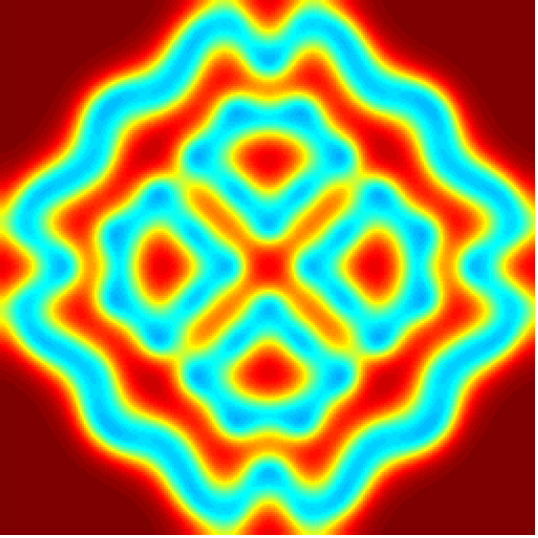}
		\end{subfigure}
		\medskip 
		
		\hfill
		\begin{subfigure}[c]{1\textwidth}
			\centering
			\begin{minipage}[t]{0.02\textwidth}
				\rotatebox{90}{\textbf{Moving spots}}
			\end{minipage}
			\includegraphics[scale = 0.4]{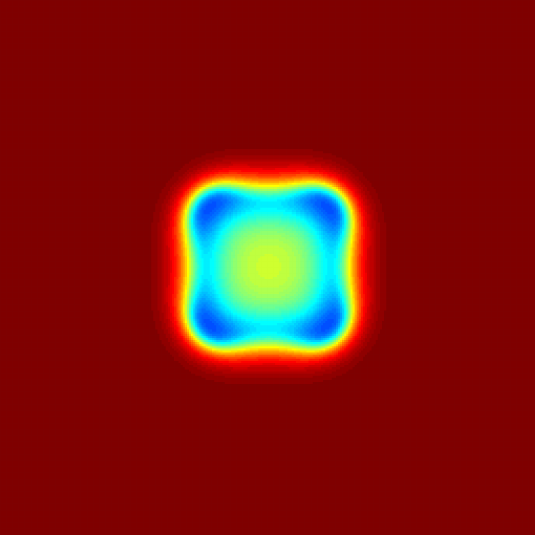}
			\includegraphics[scale = 0.4]{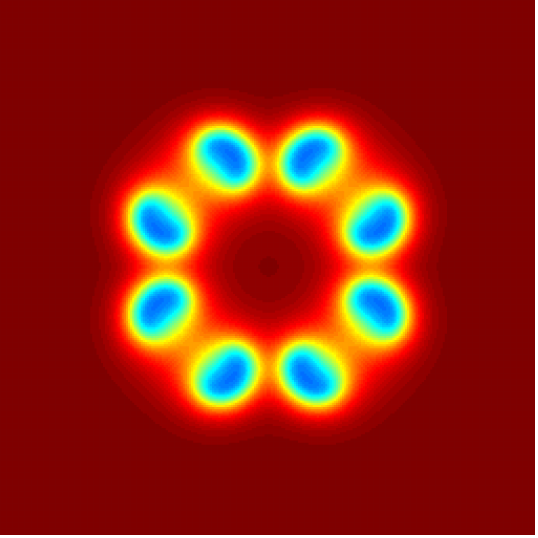}
			\includegraphics[scale = 0.4]{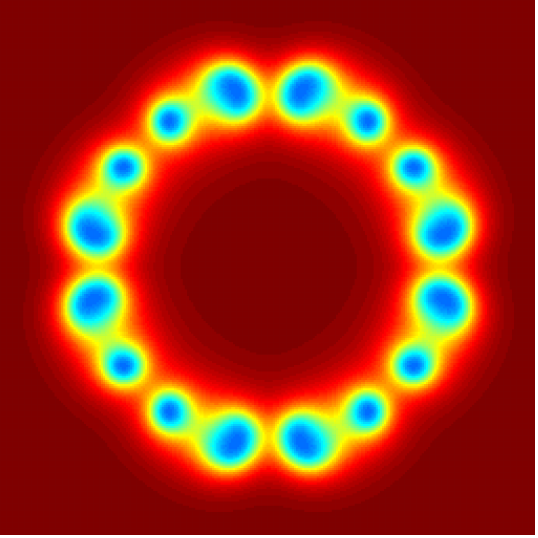}
			\includegraphics[scale = 0.4]{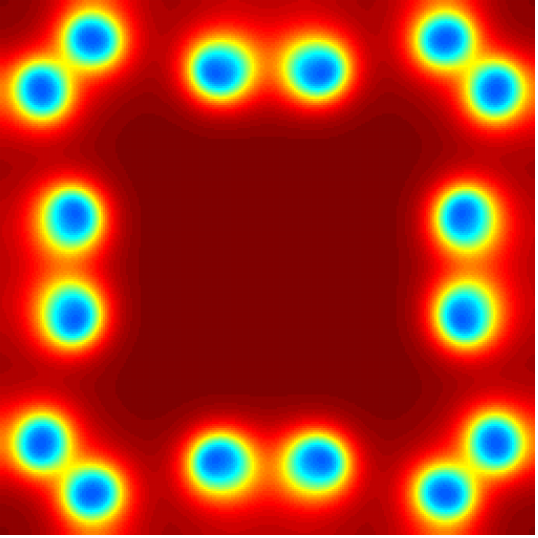}
		\end{subfigure}
		\medskip
		
		\hfill
		\begin{subfigure}[c]{1\textwidth}
			\centering
			\begin{minipage}[t]{0.02\textwidth}
				\rotatebox{90}{\textbf{Pulsating spots}}
			\end{minipage}
			\includegraphics[scale = 0.4]{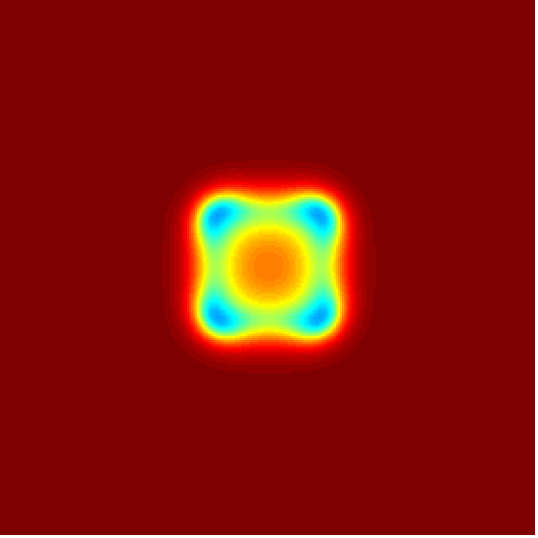}
			\includegraphics[scale = 0.4]{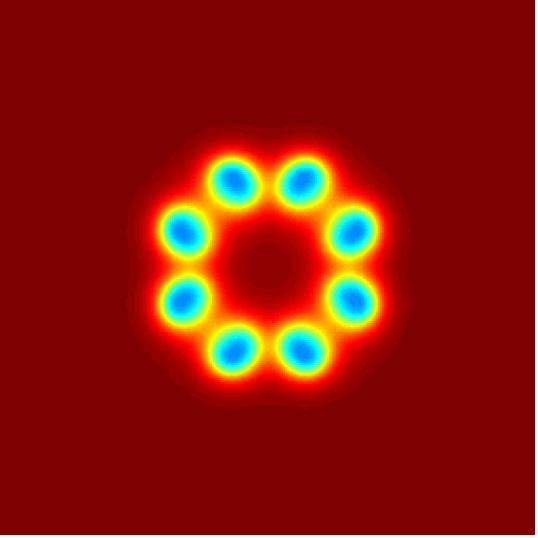}
			\includegraphics[scale = 0.4]{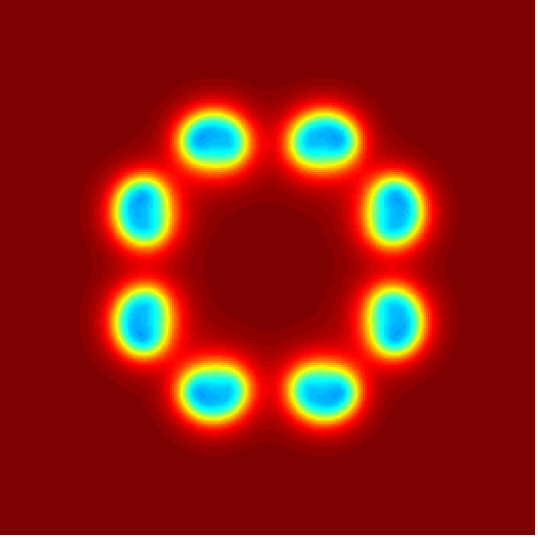}
			\includegraphics[scale = 0.4]{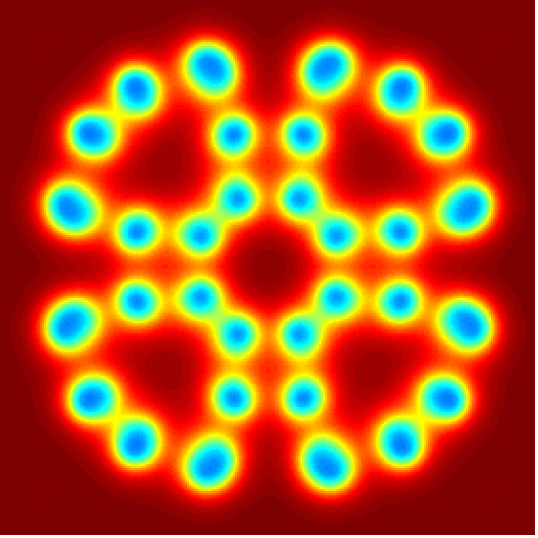}
		\end{subfigure}
		\medskip
		
		\hfill	
		\begin{subfigure}[c]{1\textwidth}
			\centering
			\includegraphics[scale = 1]{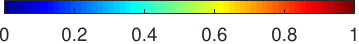}
		\end{subfigure}
		\caption{Snapshots of Gray--Scott patterns: {Labyrinthine} ($F = 0.037$, $k = 0.060$), {Moving Spots} ($F = 0.014$, $k = 0.054$), and {Pulsating Spots} ($F = 0.025$, $k = 0.060$). The system starts with $u_0 = 1$ and $v_0 = 0$ except in the center, in a box of size 0.2x0.2, where we prescribe $u_0 = 0.5$ and $v_0 = 0.25$. Colormap shows concentration $u$.}
		\label{fig:uref}
	\end{figure}
	
	In this section, we present numerical simulations that illustrate the effectiveness of the numerical scheme under consideration and validate the theoretical results on convergence.
	
	We consider the Gray--Scott system on a square domain $D = [0, 1]\times [0, 1]$. We fix the diffusion coefficients set to $d_u = 1.6\times 10^{-5}$ and $d_v = d_u/2$. 
	
	We apply \cref{algo:scheme} with a grid spacing $h$, yielding $nx = ny = 1/h$ points per spatial direction. The implementation of the numerical scheme is based on the finite volume package FiPy, see \cite{fipy}, and a semi-implicit scheme with a time step $\Delta t$.
	
	Since the Gray--Scott system does not admit a known closed-form solution, direct error evaluation against an exact solution is not possible. To assess the convergence properties of our method, there are two approaches:
	
	\textit{Reference solution approach.} We compute a high-fidelity numerical solution on a very fine spatial mesh with a very small time step, which we treat as the ``exact'' reference solution for error estimation. However, the Gray--Scott system is well known for its slow dynamics: patterns emerge only gradually, and transients can appear very similar at earlier times (e.g., around $t=100$), while the system typically settles into its characteristic long-term structures only at much later times (e.g., $t \geq 2000$, see \cref{fig:uref}). Computing such high-fidelity solutions up to $t=2000$ or beyond is computationally demanding, which motivates the use of an alternative strategy.
	
	\textit{Manufactured solution approach.} We prescribe smooth artificial solutions $u$ and $v$, and modify the Gray--Scott equations by adding source terms so that the prescribed functions are exact solutions. This provides a controlled setting in which the numerical error can be measured directly against the known exact solution.

	\subsection{Manufactured solution approach}\label{subsec:manufactured_solution}
	
	We define the modified equations 
	\begin{equation} \label{eq:manufactured_gs}
		\left\{
		\begin{alignedat}{6}
			\partial_t u &= d_u \Delta u &&- uv^2 &&+ F(1 - u) && + S_u, 
			\\
			\partial_t v &= d_v \Delta v &&+ uv^2 &&- (F + k)v &&+ S_v.
		\end{alignedat}
		\right.
	\end{equation}
	We use the parameters of the Labyrinthine ($F = 0.037$, $k = 0.060$) and fix a terminal time $T = 10$.
	The manufactured forcing terms $S_u, S_v$ are chosen so that a prescribed pair $(u^*,v^*)$ is an exact solution to \eqref{eq:manufactured_gs}. In particular,
	\begin{alignat*}{7}
		&S_u &{}={} &\partial_t u^* &- d_u \Delta u^* &+ u^*(v^*)^2 &- F(1 - u^*),&
		\\
		&S_v &{}= {}&\partial_t v^* &- d_v \Delta v^* &- u^*(v^*)^2 &+ (F+k)v^*.&
	\end{alignat*}
	
	\subsection*{Experiment 1 (trigonometric solution)}
	Let $a\in (0,1)$, $\alpha = 2\pi$, and $\omega = 2\pi$. Define
	\begin{align*}
		u^*(t,x,y) &= 1 - a\cos(\alpha x)\cos(\alpha y)\cos(w t),\quad
		v^*(t,x,y) = \frac14 +\frac14\cos(\alpha x)\cos(\alpha y)\cos(w t),
	\end{align*}
	The associated source terms are given by
	
	\begin{alignat*}{4}
		S_u(t,x,y) ={}&\ a \,\omega \cos(\alpha x)\cos(\alpha y)\sin(\omega t)
		\ {-\ } 2a\,d_u\,\alpha^2\,\cos(\alpha x)\cos(\alpha y)\cos(\omega t)
		\\[-2pt]
		&\quad {}- \Big[\hphantom{-}\,F\big(1 - u^*(t,x,y)\big) - u^*(t,x,y)\,\big(v^*(t,x,y)\big)^2 \Big],
		\\[6pt]
		S_v(t,x,y) ={}&\ {-\ }\tfrac{1}{4}\,\omega \cos(\alpha x)\cos(\alpha y)\sin(\omega t)
		\ +\ \tfrac{1}{2}\,d_v\,\alpha^2\,\cos(\alpha x)\cos(\alpha y)\cos(\omega t)
		\\[-2pt]
		&\quad {}- \Big[ - (F+k)\,v^*(t,x,y) + u^*(t,x,y)\,\big(v^*(t,x,y)\big)^2\Big].
	\end{alignat*}
	
		\begin{figure}[t]
		\begin{minipage}[c]{0.02\textwidth}
			\centering
			\rotatebox{90}{ \textbf{Experiment 2}}
			\\
			\vspace{100pt}
			\rotatebox{90}{ \textbf{Experiment 1}}
		\end{minipage}%
		\begin{minipage}[c]{0.9\textwidth}
			\begin{subfigure}[t]{0.32\textwidth}
				\centering
				\caption*{\textbf{Stability}}
				\includegraphics[scale = 0.45]{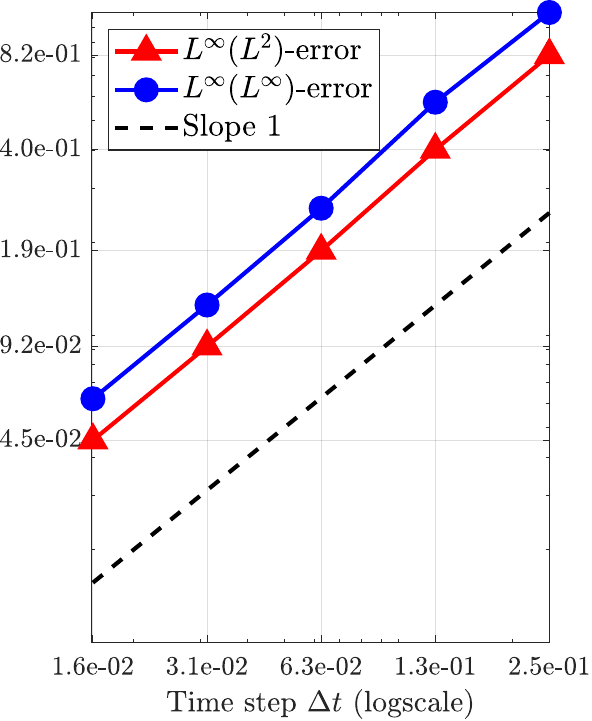}
			\end{subfigure}
			\begin{subfigure}[t]{0.32\textwidth}
				\centering
				\caption*{\textbf{Convergence}}
				\includegraphics[scale = 0.45]{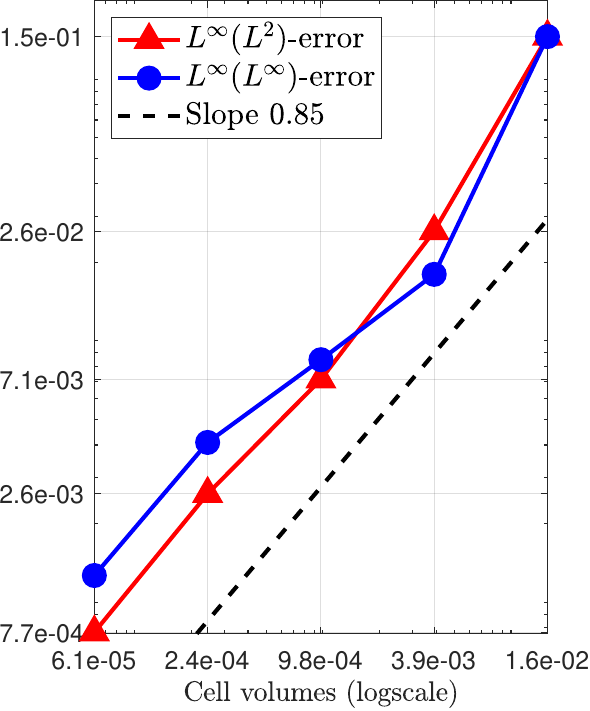}
			\end{subfigure}
			\begin{subfigure}[t]{0.32\textwidth}
				\centering
				\caption*{\textbf{Interface sensitivity}}
				\includegraphics[scale = 0.45]{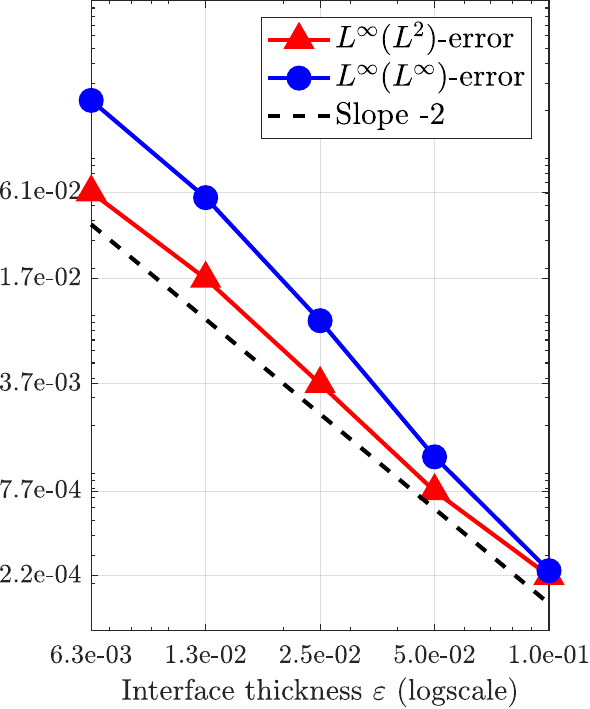}
			\end{subfigure}
			
			\begin{subfigure}[t]{0.32\textwidth}
				\centering
				\includegraphics[scale = 0.45]{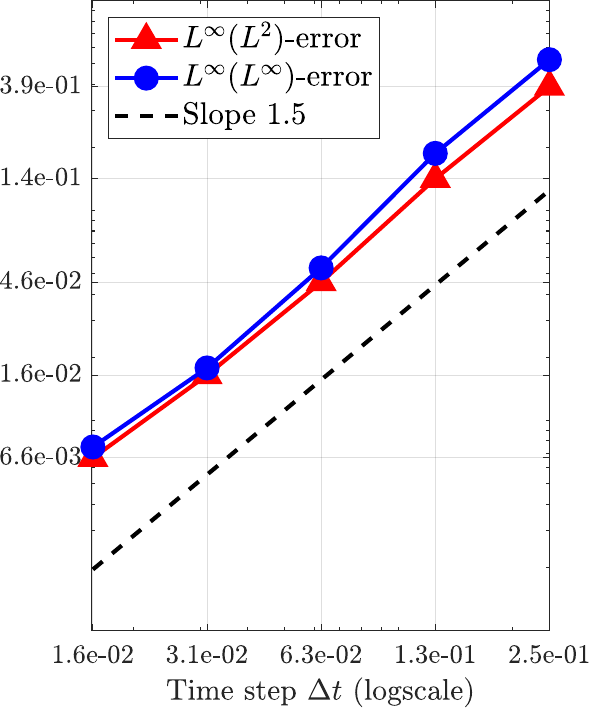}
			\end{subfigure}
			\begin{subfigure}[t]{0.32\textwidth}
				\centering
				\includegraphics[scale = 0.45]{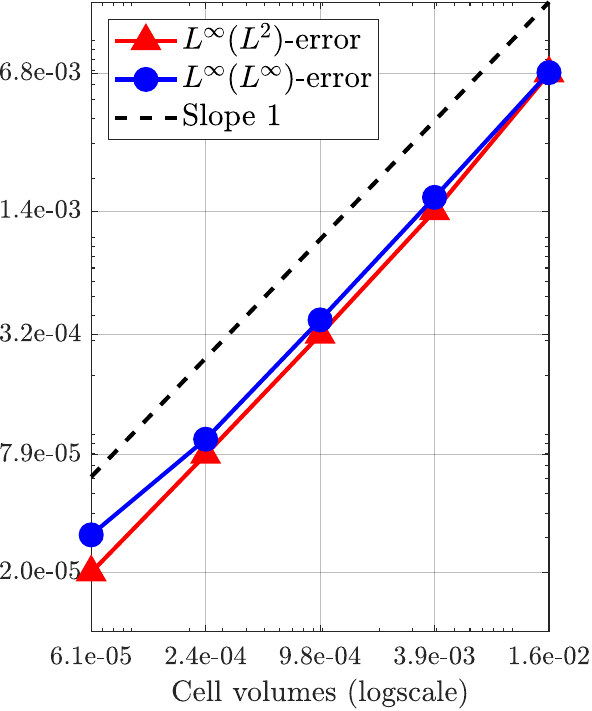}
			\end{subfigure}
			\phantom
			{%
				\begin{subfigure}[t]{0.32\textwidth}
					\centering					
					\includegraphics[scale = 0.4]{MMS2_interface.pdf}
				\end{subfigure}
			}
		\end{minipage}
		\caption{Stability, convergence, and interface sensitivity tests for Experiments 1 (top row) and 2 (bottom row). The	$L^\infty(L^2)$ and $L^\infty(L^\infty)$ errors are plotted against the cell volume $h^2$ (log--log scale), while interface sensitivity is measured with respect to the prescribed interface thickness $\varepsilon$. Reference slopes are included for comparison.}
		\label{fig:convergence}
	\end{figure}

	\subsection*{Experiment 2 (moving tanh interfacek)}
	Let $r(x, y) = \cos(2\pi (x-1/2)) + \cos(2\pi (y-1/2))$ and a time-dependent radius $r_0(t) = r_{00} + A\sin(\lambda t)$. For $\varepsilon >0$, define
	\begin{align*}
		u^*(t,x,y) = \frac12\bigg[1 + \tanh\bigg(\frac{s(t,x,y)}{\varepsilon}\bigg)\bigg],
		\quad
		v^*(t,x,y) = 1 - u^*(t,x,y).
	\end{align*}
	With $s(t,x,y) = r_0(t) - r(x,y)$. This solution has a moving discontinuity. By decreasing $\varepsilon$, the interface becomes sharper and approximates a true Heaviside jump while the derivatives (and thus sources) remain finite.	
	
	Define 
	\[
	\dot r_0(t) = A \lambda \cos(\lambda t), \quad\vert\nabla r\vert ^2 = \pi^2 \big( \sin^2(\pi x) + \sin^2(\pi y) \big),
	\quad
	\Delta r = -\pi^2 r.
	\]
	
	Plugging these into \eqref{eq:manufactured_gs} gives the associated source terms
	\begin{alignat*}{6}
		S_u(t,x,y)
		={}&&&\frac{1}{2\varepsilon} \sech^2(s)\,\dot r_0(t)
		\\
		-d_u\, \bigg[&&-&\frac{1}{\varepsilon^2}\sech^2(s)\tanh(s)\vert\nabla r\vert^2
		&&-\frac{1}{2\varepsilon}\sech^2(s)\Delta r
		&&\bigg]
		-\Big[&F\big(1-u^*\big)-&u^*(v^*)^2\Big],
		\\[5pt]
		S_v(t,x,y)
		 ={}&&-&\frac{1}{2\varepsilon} \sech^2(s)\, \dot r_0(t) 
		 \\
		 -d_v\,\bigg[&&&\frac{1}{\varepsilon^2}\sech^2(s)\tanh(s)\vert\nabla r\vert^2
		&&+\frac{1}{2\varepsilon}\sech^2(s)\Delta r&&\bigg]-\Big[-&(F+k)v^*+&u^*(v^*)^2\Big].
	\end{alignat*}
	
The $L^\infty(L^2)$- and $L^\infty(L^\infty)$-errors are computed and reported in \cref{fig:convergence}. In the convergence tests, we set $\Delta t = h^2$ (proportional to the cell volume) and gradually refined the mesh, measuring the error at the discrete times $t = 1,\ldots,10$. To assess stability, we fixed $h = 1/128$ and considered increasingly large time steps $\Delta t = k h$ with $k = 1,2,4,16,32,64$. The results confirm the unconditional stability of the scheme and verify the convergence rate of order $\approx 1$ in the $L^\infty(L^2)$ norm predicted by \cref{thm:error_estimate}. The same rate of convergence is also observed in the $L^\infty(L^\infty)$ norm for both experiments.

The interface sensitivity, relevant only for Experiment~2, exhibits second-order behavior. This demonstrates that the method remains robust even in the presence of sharp gradients induced by the moving interface.

Taken together, these results provide strong numerical evidence that the semi-implicit treatment of diffusion, combined with the explicit handling of nonlinear reactions, yields a stable and accurate discretization, fully consistent with the theoretical stability guarantees established in \Cref{sec:main_results}.

	%
	%
	\section{Conclusion and Outlook}\label{sec:conclusion}
	
	In this work, we have developed and analyzed a semi-implicit finite volume scheme for the Gray--Scott reaction-diffusion system. The scheme treats diffusion implicitly and nonlinear reaction terms explicitly, yielding a robust IMEX formulation that is well-suited for stiff dynamics. We established unconditional well-posedness of the discrete problem, together with qualitative properties such as non-negativity and boundedness of the numerical solution. By combining compactness arguments with weak--strong uniqueness, we proved that the fully discrete solution sequence converges strongly to a weak solution of the continuous system. Under additional smoothness assumptions, we derived a priori error estimates in the $L^2$-norm. Numerical experiments confirmed the theoretical results, illustrating both convergence of the method and the emergence of classical Gray--Scott pattern types.
	
	Several directions remain open for future research. An immediate extension is to apply the present analysis to other reaction-diffusion systems that exhibit rich pattern formation phenomena similar to the Gray--Scott model. Classical examples include the Schnakenberg system, the Gierer--Meinhardt activator--inhibitor model, and the FitzHugh--Nagumo system, all of which generate Turing-type patterns and have been extensively studied in mathematical biology. The analytical framework developed here, based on compactness arguments and weak--strong uniqueness, can be adapted to such systems under similar structural conditions on the nonlinearities. Beyond chemical kinetics, these models also arise in morphogenesis, neuroscience, and materials science, making them natural candidates for extending the present finite volume analysis.

	Overall, the present contribution demonstrates that finite volume schemes offer a mathematically rigorous and computationally effective tool for simulating the Gray--Scott system. The combination of provable convergence, error control, and pattern-resolving capability makes them a promising candidate for future studies in computational mathematics, applied sciences, and engineering contexts where pattern formation plays a central role.

\bibliographystyle{abbrv}

\end{document}